\documentclass{amsart}
\usepackage[mathscr]{eucal}

\title[Projection Decomposition]{Projection decomposition
in multiplier algebras}

\author{Victor Kaftal}

\address{Department of Mathematics\\
University of Cincinnati\\
P. O. Box 210025\\
Cincinnati, OH\\
45221-0025\\
USA}

\email{victor.kaftal@UC.Edu}

\author{P. W. Ng}

\address{Department of Mathematics\\
University of Louisiana\\
217 Maxim D. Doucet Hall\\
P.O. Box 41010\\
Lafayette, Louisiana\\
70504-1010\\
USA}

\email{png@louisiana.edu}

\author{Shuang Zhang}

\address{Department of Mathematics\\
University of Cincinnati\\
P.O. Box 210025\\
Cincinnati, OH\\
45221-0025\\
USA}

\email{zhangs@email.uc.edu}

\newtheorem{thm}{Theorem}[section]

\newtheorem{lem}[thm]{Lemma}
\newtheorem{cor}[thm]{Corollary}
\newtheorem{prop}[thm]{Proposition}
\newtheorem{rem}[thm]{Remark}

\newcommand{\B}{\mathscr{B}}

\newcommand{\A}{\mathscr{A}}

\newcommand{\C}{\mathscr C}

\newcommand{\h}{\mathscr H}

\newcommand{\M}{\mathbb M}

\newcommand{\K}{\mathscr K}

\newcommand{\Mul}{\mathscr M}

\newcommand{\tr}{\text{tr}}
\def\sideremark#1{\ifvmode\leavevmode\fi\vadjust{\vbox to0pt{\vss
\hbox to 0pt{\hskip\hsize\hskip1em
\vbox{\hsize2cm\tiny\raggedright\pretolerance10000
\noindent#1\hfill}\hss}\vbox to8pt{\vfil}\vss}}}

\def \bib(#1;#2;#3;#4;#5;#6)  {{#1},{\it #2} {#3},
{\bf#4} (#5) {#6}\par\smallskip}

\begin{document}

\begin{abstract} In this paper we present new structural information
 about the multiplier algebra $\Mul (\A )$ of a $\sigma$-unital purely
infinite simple $C^*$-algebra $\A $, by characterizing the positive elements $A\in \Mul (\A )$
 that are strict sums of projections belonging to $\A $.  If $A\not\in \A$ and $A$ itself is
  not a projection, then the necessary and sufficient condition for $A$ to be a strict sum of
   projections belonging to $\A $ is that $\| A\|>1$ and that the essential norm $\| A\|_{ess}\ge 1$.

Based on a generalization of the Perera-Rordam weak divisibility of
separable simple $C^*$-algebras of real rank zero to all
$\sigma$-unital simple $C^*$-algebras of real rank zero, we show that
every positive element of $\A $ with norm  greater than 1 can be
approximated by finite sums of projections. Based on block
tri-diagonal approximations, we decompose   any positive  element
$A\in \Mul (\A )$ with $\| A\|>1$ and  $\| A\|_{ess}\ge 1$ into a
strictly converging sum of positive elements in $\A$ with norm
greater than 1.
\end{abstract}

\maketitle

\section{\bf Introduction and the main result }
\setcounter{footnote}{100}

 In \cite{Fillmore} Fillmore  raised the
following question:  Which positive bounded operators  on a
separable Hilbert  space  $\h$ can be written as (finite) sums of
projections?  Fillmore obtained a characterization of the finite
rank operators that are sums of projections (see \cite{Fillmore}
Theorem 1) and of the bounded operators that are  the
sums of two projections (see \cite{Fillmore} Theorem 2).

     For infinite sums of projections  with convergence in the strong
operator topology, this question arose naturally from work on frame
theory by Dykema, Freeman, Kornelson, Larson, Ordower and Weber (see
\cite{Larsonetal}).  They proved that a sufficient condition for a
positive bounded operator $A \in \mathbb{B}(\h)$ to be a (possibly
infinite) sum of projections converging in the strong operator
topology is that its essential norm $\| A \|_{ess}$ is greater than
$1$ (see \cite{Larsonetal} Theorem 2).   This result served as a
basis for further work by Kornelson and Larson \cite{Larson1} and
then by Antezana, Massey, Ruiz and Stojanoff \cite{AMRS} on
decompositions of positive operators into strongly converging sums
of rank one positive operators with preassigned norms. In \cite
{B(H)case}, the necessary and sufficient condition for a positive
bounded operator to be  a strongly converging sum of projections was
obtained by the three authors of this article for the
$\mathbb{B}(\h)$ case and for the case of a countably decomposable
type III von Neumann factor, and for the ``diagonalizable" case  of
type II von Neumann factors.

In this paper, we extend the characterization of the positive operators that
are sums of projections to the case of
bounded  module maps (with adjoints defined)  on Hilbert
$C^*$-modules, namely,  $\mathbb{B}(\h)$ is replaced by the
 multiplier algebra  $\Mul(\A)$ of $\A$. Dealing with multiplier algebras, we
replace the strong operator topology by the strict topology.   We
point out  that when $\A$ is reduced to the algebra of complex
numbers $\mathbb C$, then $\mathbb{B}(\h) = \Mul(\K)$, the multiplier
algebra of the $C^*$-algebra $\K$ of compact operators on a separable Hilbert space, and the *-strong
operator topology on $\mathbb{B}(\h) $ is precisely  the strict
topology of $ \Mul(\K)$.

In this article we generalize the main result of \cite{Larsonetal}
to certain multiplier algebras, stated as follws.

\begin{thm}\label{T:main}  Let $\A$
be a $\sigma$-unital simple purely infinite $C^*$-algebra and $A$ be
a positive element of $\Mul(\A)$. Then $A$ is a strictly converging
sum of projections belonging to $\A$ if and only if one of the following  mutually exclusive conditions hold:
\item [(i)] $\| A \|_{ess} > 1$.
\item [(ii)] $\| A \|_{ess} = 1$ and $\| A \| > 1$.
\item [(iii)] $A\in \Mul (\A )\setminus \A $ is a projection.
\item[(iv)]  $A$ is the sum of finitely many projections
belonging to $\A$.
\end{thm}
When $\A $   is unital and hence $\Mul(\A )=\A $, if a positive
element $A\in \A $ is a strictly converging sum of (nonzero)
projections belonging to $\A$, then the  sum must be finite
(Proposition \ref {lem:idealcase}), stated  as the case (iv).

The non-trivial case is thus when $A\in\Mul(\A)\setminus \A$ where
$\A $ is $\sigma$-unital but non-unital.  Notice that such a
$C^*$-algebra ($\sigma$-unital but non-unital simple purely
infinite) is necessarily stable (see \cite {Zhang8} and \cite
{Zhang2}) and has real rank zero (\cite[1.2]{ZhangRR0}).

The necessity of the conditions (i)--(iii) is given by Corollary
\ref{C:sums}. The sufficiency of (iii) being trivial, the main focus
of this paper is to prove the sufficiency of (i) and (ii).

The proof is arranged in the following way.

In section 2 we prove that all non-elementary, $\sigma$-unital, simple $C^*$-algebras
 of real rank zero are weakly divisible in the sense of Perera-Rordam in \cite{PerRor},
 thus generalizing the previous result of \cite{PerRor} from the separable category to
  the $\sigma$-unital category. This weak divisibility property and Fillmore's
  characterization of the finite rank operators in $\mathbb B(\h)$  enable us
   to approximate a positive element with a norm greater than 1 by finite sums of
projections   (Lemma \ref {lem:finitespectrum}.)

In section 3 we prove that a positive element $A$ of $\Mul (\A )$
with essential
norm $\|A\|_{ess}>1$ can be written as a strict sum of
projections in $\A$.

Section 4 deals with the crucial case when $\|A\|_{ess} =1$ and $\|A\|>1$.   We employ
a block tri-diagonal approximation and operator theory techniques to construct a strictly
converging sequence of projections $f_k\in A$ for which $\|f_kAf_k\|>1$ for all $k$
(Lemma \ref {lem:technicallemma}). From that, we decompose $A$ into a strict sum of
projections (Proposition \ref {P:ess=1}) and conclude the proof.

Aside from the works on the $\mathbb{B}(\h)$ and von Neumann factors cases that have
been mentioned above, this paper employs some previous results and ideas on the structures
of multiplier algebras of simple purely infinite  $C^*$-algebras scattered in the several
papers such as \cite{LinAF}, \cite{LinAH},
\cite{LinZhang}, \cite{PerRor}, and \cite{Zhang1} -- \cite{Zhang9}.

     The first and second named authors participated in the NSF supported Workshop in
      Analysis and Probability, Texas A \& M University, Summer 2006, where they first
      heard from David Larson about the results in \cite{Larsonetal} and \cite{Larson1}
      that stimulated this project.


The first and third author were partially supported by  grants from the
Charles Phelps Taft Research Center.

\section{\bf Weak divisibility of $\sigma$-unital $C^*$-algebras of real rank zero }
\setcounter{footnote}{200}

In this section we  show that in a $\sigma$-unital simple purely
infinite $C^*$-algebra $\A$, every positive element  with norm
greater than 1 can be approximated \textit{from underneath} by
finite sums of projections. To do so we first extend to all
non-elementary $\sigma$-unital $C^*$-algebras of real rank zero the
property of weak divisibility obtained for separable non-elementary
simple $C^*$-algebras of real rank zero  by Perera and Rordam in
\cite[5.3]{PerRor}. Recall that a $C^*$-algebra is called
non-elementary if it is neither $\K$ nor $\mathbb M_n$ for any $n$.
A $C^*$-algebra is called to be $\sigma$-unital, if it has a
strictly positive element $b$, namely,  $(b\A )^-=(\A b)^- = \A $. A
$C^*$-algebra $\B$ is  weakly divisible (\cite[5.1,5.2]{PerRor}) if
and only if  for any nonzero projection $p$ of $\B$   there exists a
unital
*-homomorphism from $\mathbb M_{2}\oplus \mathbb M_{3} $ to $p\B p$.

\begin{prop}
If $\A $ is a non-elementary $\sigma$-unital simple  $C^*$-algebra
of real rank zero, then $\A$ is weakly divisible.
\end{prop}

\begin{proof} By \cite[Lemma 5.2]{PerRor} it suffices to show that
for each nonzero projection $p$ of $\A$ there exists a unital
*-homomorphism from $\M_2 \oplus \M_{3}$
into $p\A p$. To prove this, we use the result of divisibility of
all projections in any simple $C^*$-algebra of real rank zero in
\cite[1.1]{ZhangConf1}: For each pair of projections $(q,r)$ in $\A$
and each natural number $n$ the projection $q$ can be rewritten as a
direct sum of mutually orthogonal subprojections
$$q=p_1\oplus p_2\oplus \cdots \oplus p_{2^n}\oplus r_0 $$
such that   $p_i$ is equivalent to $p_j$ for all pairs $(i, j)$ in
the sense of  Murray-von Neumann   and $r_0$ is equivalent to a
subprojection of $r$.

Applying this result to the case $q=r=p$ and $n=1$, one has
$$ p= p_1\oplus p_2 \oplus r_0$$ where $p_1$ is equivalent to $p_2$
and $r_0$ is equivalent to a subprojection  of $p_1$, say $r_1$.
Choose a partial isometry $v$ such that $p_1=vv^*$, $p_2=v^*v$, and
set  $r_2= v^*r_1v$. Then
$$p=(p_1-r_1)\oplus (p_2-r_2)\oplus r_0\oplus r_1\oplus r_2.$$
 Then $p_1-r_1$ and $p_2-r_2$ are  equivalent, and so are $r_0$, $r_1$ and $r_2$.
 This decomposition of $p$ into these five
projections leads to a unital *-homomorphism from  $\M_2 \oplus
\M_{3}$ into $p\A p$.
\end{proof}

The same idea above also proves the following lemma that will be
used as one of the technical ingredients in this article.

\begin{lem}
Let $\A$ be a non-elementary $\sigma$-unital simple $C^*$-algebra of
real rank zero. Then for every integer $n \geq 1$ and for every
nonzero projection $p $ of $ \A $ there exists a unital
$*$-embedding of $\M_{2^n} \oplus \M_{2^n+1}$ into $p  \A p$.
\label{lem:new}
\end{lem}
\begin{proof}
Applying \cite[1.1]{ZhangConf1} to the case $q=r=p$ and arbitrary
natural number $n $, one has
$$ p= p_1\oplus p_2 \oplus \cdots \oplus p_{2^n}\oplus r_0$$ where $p_i$ is equivalent to
$p_j$ for all pair $(i,j)$ and $r_0$ is equivalent to a
subprojection  of $p_1$, say $r_1$. For each $k$ choose a partial
isometry $v_k\in p\A p$ such that $p_1=v_kv_k^*$ and $p_k=v_k^*v_k$.
Let $r_k= v_k^*r_1v_k$. Then
$$p=(p_1-r_1)\oplus (p_2-r_2)\oplus \cdots   \oplus
(p_{2^n}-r_{2^n})\oplus r_0\oplus r_1\oplus r_2\oplus \cdots \oplus
r_{2^n}.$$
 Then $p_i-r_i$ and $p_j-r_j$ are  equivalent for all pairs $(i,j)$, and so are $r_0$, $r_1$, $\cdots,$
   $r_{2^n}$. This decomposition of $p$
 leads to a unital *-homomorphism from  $\M_{2^n} \oplus \M_{2^n+1}$ into
$p\A p$.
\end{proof}

       We  need the following  approximation property for positive elements in a $C^*$-algebra of real rank zero.

\begin{lem}\label{lem:rr0} Let  $\C$ be  a $C^*$-algebra of
real rank zero and $c$ be any positive element in $ \C$. For
$\epsilon > 0$  there exist pairwise orthogonal projections $p_1,
p_2, ..., p_n$ in $\C$ and positive real numbers $\alpha_1,
\alpha_2, ..., \alpha_n$ such that
\item[(i)] $\|\alpha_1 p_1 + \alpha_2 p_2 + ...
+ \alpha_n p_n- c\| < \epsilon$
\item[(ii)]  $\alpha_1 p_1 + \alpha_2 p_2 + ... + \alpha_n p_n \leq c$.

\end{lem}
\begin{proof} Without loss of generality, assume that $\epsilon< 2\|c\|$. Let $g$ be the
 piecewise linear function $$g(x)   =:\begin{cases} 0\quad &0\le x\le \epsilon/2\\
x-\epsilon/2\quad &x> \epsilon/2.
\end{cases}$$ The  hereditary subalgebra of $\C$ generated by $\overline{g(c)\C g(c)}$  still
has real rank zero (\cite{BP}).
 Thus one can find a positive element $d \in \overline{g(c)\C g(c)}$ with finite spectrum, say
$d= \alpha_1 p_1 + \alpha_2 p_2 + ... + \alpha_n p_n $, such that
$\|g(c)-d\|<\epsilon /2$. It follows that
$$\|c-\alpha_1 p_1 + \alpha_2 p_2 + ...
+ \alpha_n p_n\|\le \|c-g(c)\|+\|g(c)-d\|< \epsilon /2 + \epsilon
/2= \epsilon.$$
  The key point is to prove that $d\le c$.
Assume without loss of generality that $\C$ act faithfully and
non-degenerately on a Hilbert space $\h$. Let   $q
=:\chi_{[\epsilon/2, \infty)}(c)$. For $ \xi \in q \h$, one has that
 \begin{align*}
<(c-d)\xi,\xi>
&=<(c-g(c))\xi,\xi >+<(g(c)-d)\xi,\xi>  \\
&=\epsilon< \xi,\xi >+<(g(c)-d)\xi,\xi> \\
&\ge  (\epsilon  -\|g(c)-d\| )<\xi,\xi>\\
& \ge 0.
\end{align*}
 If $\xi \in   q^\bot \h$, then $d\xi =0$ because $d \in \overline{g(c)\C g(c)}$, and hence also
  $<(c-d)\xi,\xi> =<c\xi,\xi >\ge 0. $ Therefore, $c \geq d$, as
wanted.
  \end{proof}

We will use the following result due to Fillmore \cite[Thm.
1]{Fillmore}
 (see   also
\cite[Prop. 6]{Larsonetal}   and  \cite [2.5, 2.6]{B(H)case}.)

\begin{prop}
Let $tr$ be the natural (non-normalized) trace on the algebra $\M_n$
of $n$ by $n$ complex matrices and  $A \in \M_n$ be a positive
matrix. Then $A$ is a sum of projections in $\M_n$ if and only if
$tr(A)$ is an integer and $tr(A) \geq \textup{rank}(A)$.
\label{prop:Larsonetal}
\end{prop}

Recall  that all simple purely infinite $C^*$-algebras have real
rank zero (\cite[1.2]{ZhangRR0}). The following lemma is one of the
two central technical ingredients of this article.

\begin{lem}  Let $\A$ be a $\sigma$-unital   purely infinite simple
$C^*$-algebra and   $A \in \A $ be a positive element  with $\| A \|
> 1$. Then for every $\epsilon > 0$  there exist positive elements
$A_1, A_2 \in \A $ such that
\item [(i)]  $A = A_1 + A_2,$
\item [(ii)]  $A_1$ is the sum of finitely many projections belonging to $\A
$, and
\item [(iii)]   $\| A_2 \| < \epsilon . $
\label{lem:finitespectrum}
\end{lem}

\begin{proof}
By Lemma \ref {lem:rr0} we can assume without loss of generality that $A$ is a positive element with finite spectrum and with norm
strictly greater than one. Then there are  nonzero pairwise orthogonal
projections
  $$e_1, e_2, ..., e_m, f_1, f_2, ..., f_n \ \in \   \A
 $$ and  strictly positive real numbers  $\lambda_1,
\lambda_2, ..., \lambda_m, \mu_1, \mu_2, ..., \mu_n$
 such that
$$A = \sum_{i=1}^m \lambda_i e_i  + \sum_{j=1}^n \mu_j f_j,$$
where $1 < \lambda_i$ and $0 < \mu_j \leq 1$ for $1 \leq i \leq m$
and  $1 \leq j \leq n$. Note that  $\| A \| > 1$ implies $m \geq 1$;
but  $n=0$ is possible.

Choose $N$ large enough in the form $2^k$ such that there are
positive integers $k_i, k'_i, l_j, l'_j$  for all $1\le i\le m$ and
$1\le j\le n$ satisfying the following inequalities:
\begin{align*}
&1<k_i/ N < \lambda_i \ \ \text{and} \ \ 1< k'_i/ (N+1) < \lambda_i ,\\
&l_j / N < \mu_j \ \ \ \text{and} \ \ \ l'_j / (N + 1) <\mu_j,\\
&0<\lambda_i - k_i/N < \frac{\epsilon}{2} \ \ \text{and} \ \  0< \lambda_i
-k'_i/(N+1) < \frac{\epsilon}{2},\\
&0<\mu_j - l_j/N  < \frac{\epsilon}{2} \ \  \text{and} \ \ 0< \mu_j -
l'_j/(N+1)  <\frac{\epsilon}{2}.
\end{align*}

By Lemma \ref{lem:new} there exists for $1 \leq i \leq m$  a
unital
*-homomorphism from    $\M_N \oplus \M_{N+1}$ onto a $C^*$-subalgebra
$\B_i$ of the corner $e_i  \A   e_i$, and for $1 \leq j \leq n$
there is a unital *-homomorphism from
  $\M_N \oplus \M_{N+1}$  onto a $C^*$-subalgebra $\C_j$  of $f_j \A
  f_j$.  Notice that for a given $i$, the projections in $\B_i$ that correspond
  to the minimal projections of $\M_N$  are all mutually equivalent, but in
   general they are not comparable to the minimal projections in $\C_j$ or in $\B_{i'}$
   for $i\ne i'$ or to those in $\B_i$   that correspond to the minimal projections of
   $\M_{N+1}$. The identity of $\B_i$ is $e_i$ and the identity of $\C_j$
  is $f_j$ and
  this way, each summand $\lambda _ie_i$ is identified with a direct
  sum of two diagonal matrices, say $B_i=B_{i1}\oplus B_{i2}$ in $\B_i$, where $B_{i1}$
  is a matrix of size $N\times N$, $B_{i2}$ is a matrix of size $(N+1)\times (N+1)$, and both
have all diagonal entries $\lambda _i$.
  Similarly, each summand $\mu _jf_j$ is identified with a direct
  sum of two diagonal matrices, say $C_j=C_{j1}\oplus C_{j2}$ in $\C_j$, where $C_{j1}$
  is of size $N\times N$,  $C_{j2}$ is of size $(N+1)\times (N+1)$, and both have all diagonal entries $\mu _j$.

  Modify $B_i=B_{i1}\oplus B_{i2}$ to $B'_i=B'_{i1}\oplus
  B'_{i2}$ where $B'_{i1}$ has the same matrix units as  $B_{i1}$ but
  has all  diagonal entries    $k_i/N$ instead of  $\lambda _i$ and  $B'_{i2}$ has the same matrix units as of
 $B_{i2}$ but has all diagonal entries    $k'_i/(N+1)$ instead of  $\lambda _i$.
 Similarly, modify $C_j=C_{j1}\oplus C_{j2}$ to $C'_j=C'_{j1}\oplus C'_{j2}$
 by replacing the
 diagonal entries $\mu _j$  of $C_{j1}$ with $l_j/N$  and the
 diagonal entries $\mu _j$  of $C_{j2}$ with $l'_j/(N+1)$. Let $$A'=\sum_{i=1}^m B'_i + \sum_{j=1}^n C'_j.$$

Notice that all the matrices $B'_{i1}$ and $C'_{j1}$ have rank $N$
and all the matrices $B'_{i2}$ and $C'_{j2}$ have rank $N+1$.
The conditions defining $k_j, k'_j, l_j, l'_j$  imply:
\begin{align*}
&0\le A'\le A\ \
 \text{and}\  \| A-A'\| <\frac{\epsilon}{2}.\\
& \tr(B'_{i1})=N(k_i/N)=k_i\quad\text{and} \quad   \tr(B'_{i2})=(N+1) k'_i/(N+1)= k'_i.\\
&\tr(C'_{j1})=N(l_j/N)= l_j \quad\text{and} \quad
\tr(C'_{j2})=(N+1)l'_j/(N+1)=l'_j.
 \end{align*}

Since
$$\tr(B'_{i1})=k_i>N =\textup{rank} (B'_{i1})\ \ \text{and }\ \ \tr(B'_{i2})=k'_i>N+1 =\textup{rank}
 (B'_{i2}),$$
  by  Proposition \ref {prop:Larsonetal} each $B'_{i1}$ and $B'_{i2}$ is a sum of projections.
  If $n=0$,  $A'=\sum_{i=1}^m B'_{i1}\oplus B'_{i2}$ is a sum
  of projections and then setting $A_1=A'$ and $A_2=A-A'$ will satisfy the thesis.

 From now on assume   $n\ge 1$. Since $\sum_{i=2}^m B'_{i1}\oplus
B'_{i2}$ is a sum of projections, it is enough to prove that $B'_1+
\sum_{j=1}^n C'_j$ is also  a sum of projections. Let $e_{11}$
(resp., $e_{12}$) be the identity of the copy of $\M_N$ (resp.,
$\M_{N+1}$) in $\B_1$. Then $B'_1= \frac{k_1}{N}e_{11}+
\frac{k'_1}{N+1}e_{12}.$ Since $1<\frac{k_1}{N} $, there exists for
each $1\le j\le n$ an integer multiple of $N$, say $L_j  $, such
that
$$L_j\big(\frac{k_1}{N}-1) \geq  N-l_j .$$
For every $1\le j\le n$, let $f_{j1}$  be a minimal projection of $C'_{j1}$. Denote by $N\cdot  f_{j1}$ the identity of the copy of $\M_N$ in $\C_j$. Then $C'_{j1}= \frac{l_j}{N}(N\cdot  f_{j1})$. Since the corner
$e_{11}\A e_{11}$ of $\A$  is still simple and purely infinite, one can
recursively find $\sum_{j=1}^n L_j$ mutually orthogonal projections
in $e_{11}\A e_{11}$, where for each $1\le j\le n$,  $L_j$ of these projections are equivalent to
$f_{j1}$ and we  denote their sum  by $L_j\cdot f_{j1}$.
Then
$$
D:= \sum_{j=1}^n\frac{k_1}{N}(L_j \cdot f_{j1})+ \sum_{j=1}^n C'_{j1}= \sum_{j=1}^n\Big(\frac{k_1}{N}(L_j \cdot f_{j1})+ \frac{l_j}{N}(N\cdot  f_{j1})\Big)
$$
and for  each $j$, $ \frac{k_1}{N}(L_j \cdot f_{j1})+ \frac{l_j}{N}(N\cdot  f_{j1})$ is a matrix of size $L_j+N$ and trace $\frac{L_jk_1}{N}+ l_j\ge L_j+N
$ and hence is the sum of projections by Proposition \ref {prop:Larsonetal}.

Similarly, there exists for
each $1\le j\le n$ an integer multiple of $N+1$, say $L'_j  $, such
that
$$L'_j\big(\frac{k_1}{N+1}-1\big) \geq N+1-l'_j .$$
For every $1\le j\le n$, let $f_{j2}$ be a minimal projection of $C'_{j2}$, $(N+1)\cdot f_{j2}$ the identity of the copy of $\M_{N+1}$ in $\C_j$ $L'_j\cdot f_{j2}$ the sum of orthogonal subprojections of $e_{12}$ equivalent to $f_{j2}$ so that $L'_j\cdot f_{j2}$ are also mutually orthogonal.
Then
$$
D':= \sum_{j=1}^n\frac{k_1}{N+1}(L'_j \cdot f_{j2})+ \sum_{j=1}^n C'_{j2}= \sum_{j=1}^n\Big(\frac{k_1}{N+1}(L'_j \cdot f_{j1})+ \frac{l'_j}{N+1}((N+1)\cdot  f_{j1})\Big)$$
is  the sum of projections by the same argument as for $D$.

Finally, let $e'=e_1 -\sum_{j=1}^n(L_j \cdot f_{j1})-\sum_{j=1}^n(L'_j \cdot f_{j})$. Then
$$
B'_1+   \sum_{j=1}^n C'_j= D+D'+ \frac{k_1}{N}e'.
$$

If $e'\ne0$, by the same argument as  for the case $n=0$ one can find a sum of projections $D''$ for which
$\|D''-\frac{k_1}{N}e' \| < \frac{\epsilon}{2}.$
Then setting $$A_1:= \sum_{i=2}^m B'_{i1}\oplus B'_{i2}+ D+D'+D''$$ and $A_2:=A-A_1$ satisfies the thesis.
\end{proof}
\

\

\section{\bf  The cases $A\in \A $ and $\|A\|_{ess} >1$ }
\setcounter{footnote}{300}

We first discuss  when a positive operator $A$ in a $\sigma$-unital
simple purely infinite $C^*$-algebra $\A$ is a strict sum of
projections in $\A$.

\begin{prop}\label{lem:idealcase}
Let $\A$ be a  $\sigma$-unital  $C^*$-algebra with an approximate
identity of projections and $A $ be a positive element in $ \A $. If
$A$ is the strict sum of projections belonging to $\A $, then $A$
must be the sum of finitely many projections belonging to $\A$.
\end{prop}
\begin{proof}
We will reason by contradiction. Assume that $\{ p_k
\}_{k=1}^{\infty}$ is an infinite sequence of \emph{nonzero}
projections in $\A$ such that $A = \sum_{k=1}^{\infty} p_k$, where
the sum converges in the strict topology in $\Mul(\A)$.

Let $\{ e_n \}_{n=1}^{\infty}$ be an approximate unit for $\A$
consisting of an increasing sequence of projections. Note that such
an  increasing approximate identity of projections indeed exists in
$\A$ (\cite{Zhang5}).
  Choose an integer $N \geq 1$ such
that for all $n \geq N$, $\|A-e_nA\|< 1/2$. As a consequence,
$$
 \| (1 - e_N) A (1 - e_N) \|
 =  \| (A - e_N A)( 1- e_N) \|
 \leq   \| A - e_N A \|
 < 1/2
$$
Recall a classical result (for example, see \cite[Lemma
III.3.1]{Davidson}) that for every $0<\epsilon <1$ there exists a
$\delta > 0$
  such that $p \in \A$ with dist$(p,(1 - e_N) \A (1 - e_N)) < \delta$ implies the existence of a projection $q
\in (1 - e_N) \A (1 - e_N)$   satisfying  $\| p - q \| < \epsilon$.
Such a projection $q$ is equivalent to $p$ in $\A $. For $\epsilon =
1/2$  there exists  $\delta > 0$. Since $\sum_{k=1}^{\infty} p_k$
converges in the strict topology on $\Mul(\A )$, let $K \geq 1$ be
such that $\|p_ke_N\|< \delta/3$ for all $k \geq K$. Hence, for all
$k \geq K$,
$$\| (1 - e_N) p_k (1 - e_N) - p_k \|
\leq \| -p_k e_N - e_N p_k + e_N p_k e_N \| \leq 3 \| p_k e_N \|  <
\delta$$ Thus dist$(p_K,  (1 - e_N) \A (1 - e_N))< \delta$. It
follows from the classical result stated above that there is a
projection
\begin{equation*}
 q \in (1 - e_N) \A
(1 - e_N)\quad\text{with} \quad \| p_K - q \| < 1/2.
\end{equation*}
Now let  $B = \sum_{k \neq K} p_k  +  q$.  Then
$$\| B - A \| = \| q - p_K \| < 1/2,$$
and hence,
$$\| (1 - e_N) B(1 - e_N) - (1 - e_N) A (1 - e_N) \| < 1/2.$$
Applying  the triangle inequality, one has
 $$\| (1 - e_N) B (1 - e_N) \| < 1/2 + \| (1 - e_N) A (1 - e_N) \|< 1.$$
On the other hand, $(1 - e_N) B (1 - e_N) \geq q $ implies $\|(1 -
e_N) B (1 - e_N)\| \geq 1$, a contradiction. Therefore,   $A$, as the
strict sum of projections, must be a finite sum.
\end{proof}

   We now turn to handle the sufficient condition $\|A\|_{ess}>1$.
Let us first review some elementary facts about the essential norm,
which are  formulated only for the special cases that we will work
with. Let $\A$ be a non-unital $C^*$-algebra, let $\pi$ be the
canonical homomorphism from $\Mul(\A)$ onto the corona algebra $\Mul
(\A)/{\A}$, and for every $A\in \Mul(\A)$, let $\|A\|_{ess}:=
\|\pi(A)\|$ denote the essential norm.
\begin{lem} \label{L:ess norm}
Let $\A$ be a non-unital $C^*$-algebra.
\item[(i)]
For every positive $A\in \Mul (\A)$, $$\|A\|_{ess}= \inf \{\|A(I-a)\| \mid a\in \A^+, \|a\|\le 1\}.$$
\item[(ii)]
Let $A\in \Mul (\A)\setminus \A$ be a positive element, and let $
a_n$    be a monotone increasing sequence of positive elements of
$\A $ converging to $A$ in the strict topology. Then $$\|A\|_{ess}=
\inf_n \|A-a_n\|.$$
\end{lem}
\begin{proof}
\item [(i)]  Since $ Aa\in \A$ for every $a\in \A$, it follows that  $ \|A\|_{ess}= \|A(I-a)\|_{ess}\le  \|A(I-a)\| $ and hence $$\|A\|_{ess}\le \inf \{\|A(I-a)\| \mid a\in \A^+, \|a\|\le 1\}.$$  If $\|A\|_{ess}= \|A\|$, then the reverse inequality holds by choosing $a=0$, so assume that $\|A\|> \|A\|_{ess}$.  Let  $0< \epsilon< \|A\|- \|A\|_{ess}$, let $h$ be the positive continuous function on the interval $[0, \|A\|]$ defined as
$$
h(t):= \begin{cases}
0 &t\in [0, \|A\|_{ess}]\\
\text{linear} \quad & t\in[\|A\|_{ess},\|A\|_{ess}+\epsilon]\\
1 & t\in[\|A\|_{ess}+\epsilon, \|A\|]
\end{cases},
$$
and let $a:= h(A)$. Clearly, $a\ge 0$ and $\|a\|=1$. Via the Gelfand's transformation,  identify $C^*(\pi (A))$ with the algebra of complex-valued continuous functions
$C(\sigma _e(A))$ defined on the essential spectrum $ \sigma _e(A)$ of $A$. Since $h$ vanishes on $ \sigma _e(A)$ and $h\circ \pi = \pi\circ h$, it follows that $\pi(h(A))=0$ and hence $h(A)\in \A$. Moreover,
$$\|A(I-a)\|= \|t(1-h(t))\|_\infty \le \|A\|_{ess}+\epsilon,$$ whence
$$\inf \{\|A(I-a)\| \mid a\in \A^+, \|a\|\le 1\}\le  \|A\|_{ess}.$$
Thus  equality holds, proving (i).
\item[(ii)]
Since for every $n$
$$
\|A\|_{ess}=  \|A-a_n\|_{ess}\le \|A-a_n\|,
$$ it follows that $$ \|A\|_{ess} \le \inf_n \|A-a_n\|.$$  For every positive contraction $a\in \A$ and every $n$
$$
\|A-a_n\|^{1/2}= \|(A-a_n)^{1/2}\|\le \|(A-a_n)^{1/2}a\|+\|(A-a_n)^{1/2}(I-a)\|.
$$
Since  $0\le A-a_n\le A$,
$$
\|(A-a_n)^{1/2}(I-a)\|^2= \|(I-a)(A-a_n)(I-a)\|\le \|(I-a)A(I-a)\| = \|A^{1/2}(I-a)\|^2.$$
But then
$$
 \|A^{1/2}(I-a)\|\ge \|(A-a_n)^{1/2}(I-a)\|\ge \|A-a_n\|^{1/2}- \|(A-a_n)^{1/2}a\|.
$$
Since $A-a_n\to0$ in the strict topology it follows that $\|(A-a_n)^{1/2}a\|\to0$. Since $a_n$ is monotone increasing, it follows that $\|A-a_n\|^{1/2}\to \inf \|A-a_n\|^{1/2}$ and hence $$ \|A^{1/2}(I-a)\|\ge \inf_n \|A-a_n\|^{1/2}.$$
Thus
$$\inf \{\|A^{1/2}(I-a)\| \mid a\in \A^+, \|a\|\le 1\}\ge \inf_n \|A-a_n\|^{1/2}$$
and by (i), $$ \|A^{1/2}\|_{ess}\ge \inf_n \|A-a_n\|^{1/2}.$$
Since $\|A\|_{ess} = \|A^{1/2}\|_{ess}^2 $, it follows that
$$\|A\|_{ess}\ge \inf_n \|A-a_n\|, $$ which concludes the proof.
\end{proof}
\begin{cor}\label{C:sums}
Let $\A$ be a non-unital $C^*$-algebra and let $A=\sum_{j=i}^\infty a_j$ where $a_j\in\A^+$, $\|a_j\|\ge 1$ for all $j$  and the series converges in the strict topology of $\Mul(\A)$.  Then $\|A\|_{ess}\ge 1.$
\end{cor}
Every  $\sigma$-unital $C^*$-algebra $\A$ has a strictly positive
element $b\in  \A $, i.e.,  a positive element for which $(b\A
)^-=(\A b)^- = \A$. As usual, one can assume that $\|b\|=1$. Define
a seminorm on $\Mul(\A
 )$, say $\| . \|_b$,   by $$\| m \|_b:= \| mb \| + \| b m
\|\ \ \text{for all }\ \ m \in \Mul(\A  ).$$ Clearly, $\| . \|_b$
generates the strict topology on $\Mul(\A  )$.  Note that $\|
m\|_b\le 2\| m\| $ for all $m\in \Mul(\A )$.
\begin{prop}\label{P:ess>1}
Let $\A$ be a $\sigma$-unital non-unital   purely infinite simple
$C^*$-algebra  and let $A\in \Mul(\A)$ be a positive element with
$\|A\|_{ess}>1$. Then $A$ is a strict sum of projections.
\end{prop}
\begin{proof}
Every $\sigma$-unital, non-unital $C^*$-algebra of real rank zero
has an approximate identity of projections; such an approximate
identity can always be chosen to be  countable and increasing, say
$\{e_j\}$ (\cite{Zhang5}).

Let $q_j=e_{j}-e_{j-1}$ setting $e_o=0$. Then $\sum_{j=1}^\infty
q_j=I$, where the convergence is in the strict topology.
Furthermore,
$$
A= \sum_{j=1}^\infty A^{1/2}q_jA^{1/2}
$$
where the convergence is also in the strict topology.  By Lemma \ref {L:ess norm} (ii),
$$\|\sum_{j=n}^\infty A^{1/2}q_jA^{1/2}\|\ge \|A\|_{ess}$$
for every $n$. Thus the condition $\|A\|_{ess}>1$ allows us to find
a strictly increasing sequence of integers $n_k$ starting with
$n_0=1$ such that
$$
\|\sum_{j=n_{k-1}}^{n_{k}-1} A^{1/2}q_kA^{1/2}\|> 1$$ for every $k$.
Let
$$a_k:= \sum_{j=n_{k-1}}^{n_{k}-1} A^{1/2}q_kA^{1/2}.$$
Then $a_k$ is a positive element in $ \A^+$ with $\|a_k\|>1$ for
every $k$ and $A= \sum_{k=1}^\infty a_k$ in the strict topology.
Thus
$$
\| \sum_{k=n}^\infty a_k\|_b\to0.
$$
Apply Lemma \ref{lem:finitespectrum} to $a_1$ to obtain a finite sum of projections $d_1\in \A$, $d_1\le a_1$ with
 $$\|d_1-a_1| < \frac{1}{2}\|\sum_{k=2}^\infty a_k\|_b.$$
Let $b_1:=a_1-d_1\in \A^+$  and hence
$\|b_1\|_b\le\|\sum_{k=2}^\infty a_k\|_b$. Then  $A-d_1= b_1+
\sum_{k=2}^\infty a_k$, and hence,
$$
\|A-d_1\|_b\le \|b_1\|_b + \|\sum_{k=2}^\infty a_k\|_b\le 2\|\sum_{k=2}^\infty a_k\|_b.
$$
 Next, since $b_1+a_2\in \A^+$ and $\|b_1+a_2\|\ge \|a_2\|>1$, we can apply Lemma \ref {lem:finitespectrum} to $b_1+a_2$ to obtain a
 finite sum of projections $d_2\le b_1+a_2$ with $$\|b_1+a_2-d_2\|\le \frac{1}{2}\|\sum_{k=3}^\infty a_k\|_b.$$
  Thus, iterating, we can find for each $k$ a finite sum $d_k$ of  projections in $\A$ so that
  $$
\| A-\sum_{k=1}^nd_k \|\le 2 \|\sum_{k=n+1}^\infty a_k\|_b \to 0.
$$
  This proves that the sum $ \sum_{k=1}^\infty d_k$ converges to $A$ in the strict topology,
  and hence that $A$ is a strict sum of projections, as claimed.
  \end{proof}

  \begin{rem}\label{R:sums}
  In the course of the above proof we have proven that if $\A$ is a $\sigma$-unital
  non-unital   purely infinite simple
  $C^*$-algebra, and $A=\sum _{k=1}^\infty a_k$ in the strict topology,
   where $a_k\in \A^+$ and $\|a_k\|>1$ for all $k$, then $A$ is a strict sum of
   projections. The condition "purely infinite and simple" is the
   key assumption for the conclusion to hold in the eyes of key Lemma 2.5.
  \end{rem}

\section{\bf The case $\|A\|_{ess}=1$ and $\|A\|>1$ }
The objective of this section is to prove that $\|A\|_{ess}=1$ and
$\|A\|>1$ suffice to have $A$ written as a strictly converging sum
of projections in $\A $. We start with some technical preparations.

\begin{lem}  Let $\A$ be any $C^*$-algebra of  real rank zero
and $A$ be a positive element in $\Mul(\A  )$ such that $\| A
\|_{ess} = 1$ and $\| A \| > 1$. Then there exist a positive element
$A'\in \Mul(\A  )$, a real
number $\lambda
> 1$, and a nonzero projection $p\in \A  $   such that
\item [(i)] $\| A' \|_{ess} =1,$
\item [(ii)] $A' p = p A' = 0,$
\item [(iii)] $A' + \lambda p \leq A.$
\label{lem:bigp}
\end{lem}

\begin{proof}
Let $\delta =  \|A\|-1$. Define two positive continuous functions $h_1(t)$ and $h_2(t)$ on $[0, \|A\|]$ as follows:

\[
h_1(t):= \begin{cases}
0\quad &t\in [0, 1+\frac{\delta }{2}]\\
\text{linear} &t\in [1+\frac{\delta }{2},1+\frac {3\delta }{4}]\\
t & t\in [1+\frac{3\delta }{4}, \|A\|]
\end{cases}\qquad\text{and} \quad  h_2(t):= \begin{cases}
t\quad &t\in [0, 1+\frac{\delta }4]\\
\text{linear} &t\in [1+\frac{\delta }4,1+\frac {\delta }2]\\
0 &t\in [1+\frac {\delta }2, \|A\|]
\end{cases}
\]

Clearly, $h_1(t)+h_2(t) \le t$ and $h_1(t)h_2(t)=0$ for all $t$, hence, $h_1(A)+h_2(A) \le A$ and $h_1(A)h_2(A)=0$.
 Let $\pi $ be the quotient map from $\Mul (\A)$ to the
corona algebra $\Mul(\A )/{\A}$.
Reasoning as in Lemma \ref{L:ess norm},  $h_1(A)\in \A$ and $ \| h_2(A)\|_{ess}=  \| A\|_{ess} =1$.
Applying Lemma \ref{lem:rr0},   approximate  $h_1(A)$
by a positive element of finite spectrum satisfying
$$\alpha _1p_1+\alpha _2 p_2+\cdots +\alpha _mp_m\le h_1(A)$$  where $p_i$ are pairwise orthogonal nonzero
projections in $\A$.  For a sufficient approximation, $\alpha _i>1$ holds for at least one $i_0$. Set $\lambda := \alpha _{i_0}$,  $p:=p_{i_0}$, and
$A'=h_2(A).$ Then (i) is satisfied. Since $\lambda p\le h_1(A)$ and hence  $A'+\lambda p \le A$, i.e., (iii) is satisfied. Since $h_1(A)A'=A' h_1(A)=0$, it follows that $A'p=pA'=0$. i.e.,  (ii) is satisfied.
\end{proof}

 The content of the following lemma  can be found in the proof of
Theorem 2.2 of \cite{ZhangSimple}.

\begin{lem}  Let $\A$ be a $\sigma$-unital, non-unital $C^*$-algebra that has
an approximate identity of projections. If  $A \in \Mul(\A  )$ is a
positive element, then for every $\epsilon
> 0$ there exist three positive
elements $A_1, A_2, A_3 \in \Mul(\A )$ and there is a self-adjoint
element $a \in \A $ with $\| a \| < \epsilon$ such that
$$A = A_1 + A_2 + A_3 + a,$$
where all $A_1, A_2, A_3 $ are in block-diagonal forms (see the
detailed descriptions in the following proof).
\label{lem:threediagonals}
\end{lem}
\begin{proof} The details were given in the proof of
\cite[2.2]{ZhangSimple}, but we sketch them here for the convenience
of the readers.  Let $ \{e_j\}$ be an approximate unit of
projections.  For $i\ge 1$ we will view $(e_i-e_{i-1})A^{\frac
12}(e_j-e_{j-1}) $ as the $(i,j)$-entry of $A^{\frac 12}$, and view
$(e_{n_{i+1}}-e_{n_i})A^{\frac 12}(e_{n_j}-e_{n_{j-1}})$ as the
$(i,j)$-block entry of $A^{\frac 12} $.

Using a standard argument recursively on $A^{\frac 12}$ one can find
an increasing sequence of indices $\{n_i\}$ starting with $n_0=0$ such that $A^{\frac 12}$
can be rewritten as a sum $A^{\frac 12}=: X+a$ of two self-adjoint
elements, where setting $e_{n_0}=0$,
$$a=: \sum_{i=1}^\infty \{
(e_{n_i}-e_{n_{i-1}})A^{\frac
12}(1-e_{n_{i+1}})+(1-e_{n_{i+1}})A^{\frac
12}(e_{n_i}-e_{n_{i-1}})\}
 $$ satisfies  $\|a\|\le \frac{\epsilon}{2\sqrt{\|A\|}+1}$   and
$X=:A^{\frac 12}-a$.  Then
 \begin{align*}
 X=&\sum_{i=1}^\infty (e_{n_{i+1}}-e_{n_i})A^{\frac 12}(e_{n_i}-e_{n_{i-1}})\\
 +&\sum_{i=1}^\infty(e_{n_i}-e_{n_{i-1}})A^{\frac 12}(e_{n_i}-e_{n_{i-1}}) \\
+&\sum_{i=1}^\infty(e_{n_i}-e_{n_{i-1}})A^{\frac
12}(e_{n_{i+1}}-e_{n_i}),\end{align*} the second sum above can be
viewed as the main block diagonal,  the first sum (resp., last sum)
can be viewed as the first  block diagonal  below (resp., above) the
main one. In this way, $X$ is said to have a tri-block diagonal
form.

Define  \begin{align*}
& A_1=: X\sum_{i=1}^\infty (e_{n_{3i-2}}-e_{n_{3i-3}}) X,\\
& A_2=: X\sum_{i=1}^\infty (e_{n_{3i-1}}-e_{n_{3i-2}}) X,\\
& A_3=: X\sum_{i=1}^\infty (e_{n_{3i}}-e_{n_{3i-1}}) X.
\end{align*}
Clearly, $A_1+A_2+A_3=X^2= A  -A^{\frac 12}a_0-a_0A^{\frac
12}+a_0^2$ and  all three sums $A_1$, $A_2$, and $A_3$ strictly
converge to positive elements of $\Mul (\A)$. Set $a=A^{\frac
12}a_0+a_0A^{\frac 12}-a_0^2$. Then $\|a\|< \epsilon$.

Via multiplication one sees that $A_1$ is of block diagonal with
respect to the decomposition of the identity (of $\Mul (\A )$)
$$ 1 =  e_{n_3} \oplus (e_{n_6} -e_{n_3})\oplus \cdots
\oplus (e_{n_{3i} }-e_{n_{3i-3)}})\oplus \cdots ,$$ $A_2$ is of
block diagonal with respect to the decomposition
$$ 1  =  (e_{n_4}-e_{n_1}) \oplus (e_{n_7} -e_{n_4})\oplus \cdots
\oplus (e_{n_{3i+1}} -e_{n_{3i-2}})\oplus \cdots ,$$ and $A_3$ is of
block diagonal with respect to the decomposition
$$ 1 = (e_{n_5}-e_{n_2})\oplus (e_{n_8} -e_{n_4})\oplus \cdots
\oplus (e_{n_{3i+2}} -e_{n_{3i-1}})\oplus \cdots .$$
\end{proof}

\begin{lem}  Let $\A$ be a $\sigma$-unital, non-unital
$C^*$-algebra of real rank zero and  let $A \in \Mul(\A )\setminus \A$ be a positive element.
Then for every $\epsilon
> 0$ there exist a sequence $\{ q_k \}_{k=1}^{\infty}$ of pairwise
orthogonal nonzero projections in $\A$, a bounded sequence $\{
\lambda_k \}_{k=1}^{\infty}$ of positive real numbers, a positive
element $A_0\in \Mul(\A  ) $,   and a self-adjoint element $a \in
\A $ such that the following hold:
\item [(i)] $\| a \| < \epsilon.$
\item  [(ii)]$\sum_{k=1}^{\infty} q_k$ converges in the strict topology of
$\Mul(\A ).$
\item  [(iii)]$A = A_0 + \sum_{k=1}^{\infty} \lambda_k q_k + a.$
\item  [(iv)] $\underset{k \to \infty} {\lim } \lambda_k = \| A
\|_{ess}$.
 \label{lem:diagonalize}
\end{lem}

\begin{proof} Applying Lemma \ref{lem:threediagonals},  one has a decomposition
$$A = A_1 + A_2 + A_3 + a$$ where $a \in \A $ is self-adjoint,  $\| a \| < \epsilon$,  $A_1, A_2, A_3 \in
\Mul(\A  )$ are positive element in block-diagonal forms, as
described in the proof of \ref{lem:threediagonals}. Let  $a_{i,j}$
be the  jth-block on the diagonal of $A_i$ for $i=1,2,3$ and $j=1,
2,\cdots$.  All $a_{i,j}$ are positive elements in $ \A  $ and
$\sum_{i = 1}^3 \sum_{j = 1}^{\infty} a_{i,j}= A_1+A_2+A_3$
converges in the strict topology.

We construct by induction a 
 sequence of
positive numbers $\lambda_k $ and mutually orthogonal projections
$q_k \in \A  $ such that for each $k$
\begin{align*}
& \| A \|_{ess}\ge \lambda_k > \| A \|_{ess} -1/2^{k} \\
& \textstyle \sum_{j=1}^k \lambda_kq_k \le  \sum_{i = 1}^3 \sum_{j =
1}^{n_k} a_{i,j},
\end{align*}
where $\{n_k\}$ in an increasing sequence of natural numbers.

   For $k = 1$, since  $$\| \sum_{i = 1}^3 \sum_{j = 1}^{n} a_{i,j} \| \uparrow  \|A-a\|\ge  \|A-a\|_{ess} = \|A\|_{ess},$$
   we can choose an integer $n_1$ such that
$\| \sum_{i = 1}^3 \sum_{j = 1}^{n_1} a_{i,j} \| > \| A \|_{ess} -
1/2$. By  Lemma \ref{lem:rr0} applied to $ \sum_{i = 1}^3 \sum_{j =
1}^{n_1} a_{i,j}\in \A $  one can find an approximation of $ \sum_{i
= 1}^3 \sum_{j = 1}^{n_1} a_{i,j}$ by a positive element of finite
spectrum belonging to $ \A  $,   $\sum_{i=1}^m\alpha _ip_i$, with
$\alpha_i>0$, $p_i$ mutually orthogonal nonzero projections of  $ \A
$, and
 \begin{align*} &\sum_{i=1}^m\alpha _ip_i\le \sum_{i = 1}^3 \sum_{j = 1}^{n_1}
a_{i,j},\\
&\|\sum_{i = 1}^3 \sum_{j = 1}^{n_1} a_{i,j}-\sum_{i=1}^m\alpha
_ip_i\|\le\frac 12 ( \| A \|_{ess} -1/2) \|.\end{align*}  Then at
least one of $\alpha _1, \alpha _2, \cdots, \alpha _m$, say $\alpha
_j$, satisfies
$$\alpha_j > \| A \|_{ess} -1/2.$$
Let $\lambda _1=:\min\{\alpha _j, \|A\|_{ess}\}$ and $q_1=:p_j$, where $q_1$ is a
nonzero projection from $ \A $. Then one has the desired inequality:
$$\lambda_1 q_1 \leq \sum_{i = 1}^3 \sum_{j = 1}^{n_1} a_{i,j}.$$

For $k=2$  take  an integer $m\ge n_1+3$ so that the sum
$\sum_{i=1}^3 \sum_{j = m+1 }^{\infty } a_{i,j}$ is orthogonal to
$\sum_{i = 1}^3 \sum_{j = 1}^{n_1}
a_{i,j}$ in the sense
$$(\sum_{i=1}^3 \sum_{j = m+1 }^{\infty } a_{i,j})(
  \sum_{i = 1}^3 \sum_{j = 1}^{n_1}a_{i,j})= (
 \sum_{i = 1}^3 \sum_{j =
1}^{n_1}a_{i,j})(\sum_{i=1}^3 \sum_{j = m+1 }^{\infty }
a_{i,j})=0.$$ Such an $m$ exists by the construction of $A_1,
A_2, A_3$.

Since $ \| \sum_{i = 1}^3 \sum_{j =m+1}^{n} a_{i,j}\|$  increases to
$\|A-a-  \sum_{i = 1}^3 \sum_{j =1}^{m} a_{i,j}\|  $ and
$$\|A-a-  \sum_{i = 1}^3 \sum_{j =1}^{m} a_{i,j}\| \ge \|A-a-
\sum_{i = 1}^3 \sum_{j =1}^{m} a_{i,j}\|_{ess}=\|A\|_{ess},
$$
repeating the above argument for $k=1$, choose an integer $n_2>m$
such that $$\| \sum_{i = 1}^3 \sum_{j = m+1}^{n_2} a_{i,j} \| > \| A
\|_{ess} - 1/2^2.$$ As for the case $k=1$, one can choose  a nonzero
projection $q_2$   in $ \A  $ and $\lambda_2 > 0$ such that
\begin{align*}&\| A \|_{ess} \ge \lambda_2
> \| A \|_{ess} - 1/ 2^2,\\
&\lambda_2 q_2 \leq \sum_{i=1}^3 \sum_{j = m+1}^{n_2}
a_{i,j}. \end{align*}

Clearly, $q_2q_1=0$ and the obvious inequality  $$\sum_{i=1}^3
\sum_{j = m+1}^{n_2} a_{i,j}\le \sum_{i=1}^3 \sum_{j = n_1+ 1}^{n_2}
a_{i,j}$$ guarantees that $$\lambda _1q_1+\lambda_2q_2 \le \sum_{i=1}^3
\sum_{j =  1}^{n_2} a_{i,j}.$$

Proceeding recursively, one constructs  a
sequence of pairwise orthogonal projections $\{q_k\}$ and a sequence
of positive numbers $\{\lambda _k\}$ with the required properties.

It is now routine to prove
 that $\sum_{k=1}^{\infty}\lambda_k q_k$ and hence also $\sum_{k=1}^{\infty} q_k$ converge in the strict topology of $\Mul(\A
 )$.
 Then setting
$$
A_o= A_1+A_2+A_3- \sum_{k=1}^{\infty}\lambda_k q_k\in  \Mul(\A)
$$
satisfies (i)--(iv).
\end{proof}

   We now reach our second key lemma.

\begin{lem} Let $\A$ be a $\sigma$-unital non-unital
purely infinite simple $C^*$-algebra and $A$   be a positive element
of $ \Mul(\A  )$ such that $\| A \|_{ess} = 1$ and $\| A \| > 1$.
Then there exists a sequence $\{ f_k \}_{k=1}^{\infty}$ of pairwise
orthogonal projections in $\A  $ such that
\item [(i)]  $\sum_{k=1}^{\infty} f_k$ converges in the strict topology
in $\Mul(\A  )$, and
\item [(ii)] $\| f_k A f_k \| > 1$ for all $k$.
\label{lem:technicallemma}
\end{lem}

\begin{proof} By Lemma \ref{lem:bigp} there is a positive element $A' $ of $\Mul(\A  )$
 with $\| A' \|_{ess} =1$, a projection $q_0$ of $ \A$ with $q_0A'  = A' q_0  = 0,$ and a scalar $\lambda_0>1$ such that
$$A_{0}:= A-A' - \lambda_0 q_0\ge 0.$$

By Lemma \ref{lem:diagonalize} applied to $A'$ and $\epsilon=1$,
there is  a sequence of positive real numbers $ \lambda'_k \to 1$,  a sequence of
pairwise orthogonal nonzero projections $\{ q'_k \}_{k = 1}^{\infty}\in\A  $,  a self adjoint
element $a=a^*\in \A$, and  a  positive element  $ A'_0\in \Mul(\A
)$ such that
$$ A'=A'_0 + \sum_{k=1}^{\infty} \lambda'_k q'_k   + a.$$
 Notice that we can choose
$q_0q'_k=q'_kq_0=0$ for all $k$ because $q_0A'=A'q_0=0$ (just replace 
$\A$ with $(1 - q_0) \A (1 - q_0)$ when applying Lemma 
\ref{lem:diagonalize}).

Choose a subsequence $\lambda_k:=\lambda'_{n_k}$
such that
\begin{equation}\label{e:lambda}| \lambda_k- 1 | < (\lambda_0 - 1)/ 4^{k+1}.\end{equation}
Since   $\sum_{k \geq 1} q'_{n_k}$ still converges in the strict
topology of $\Mul(\A )$,
 by passing if necessary to a subsequence,   one can further assume that
\begin{equation}\label{e:apm}\max\{\| (a_+)^{1/2} q'_{n_k} \|,\| (a_-)^{1/2} q'_{n_k} \| \}
 < \frac { \sqrt{\lambda_0-1} (\sqrt{2} - 1) } {2^{k +1}\sqrt{2}
 },\end{equation}
where $a_-$ and $a_+$ denote the negative and positive parts of $a$,
respectively. Since $\A$ is  simple and purely infinite, for each $k\ge 1$  there exists a subprojection $q_k$ of $q'_{n_k}$ such that
$q_k\sim q_0$.

Set \begin{align*} q&: =\sum_{k=0}^{\infty} q_k\\
 A_{00}&= A_0+A_0'+ \sum_{i\not\in \{n_ k\}} \lambda
_iq'_i+\sum_{k=1}^\infty \lambda' _{n_k} (q'_{n_k}-q_k).\end{align*}
Then $q$ is a projection of $\Mul(\A)$ (as a direct sum of countably
many mutually orthogonal copies of $q_0$), $A_{00}$ is a positive element of $\Mul(\A
)$,  and
\begin{equation}\label{e:A}A =  A_{00}+ \sum_{k=0}^{\infty} \lambda_k q_k   + a.\end{equation}
 Let $\{e_i\}_{i=0}^\infty$
be the sequence of mutually orthogonal rank-one projections in
 $\mathbb{B}( l_2 )$
corresponding to the standard basis of $\ell_2$ and let $\rho $ be
a unital (isometrical) $*$-embedding $\mathbb{B}( l_2 )
\rightarrow q \Mul(\A ) q$    for which
$$\rho(e_i) = q_i \ \ \text{
for all}\ \ \ i \geq 0.$$

It is easy to verify that the following matrix $u$ is unitary
\[\left(\begin{array}{ccccccccc}
1/\sqrt{2} & -1/\sqrt{2} & 0 & ....\\
(1/\sqrt{2})^2 & (1/\sqrt{2})^2 & -1/\sqrt{2} & 0 & ...\\
......\\
......\\
(1/\sqrt{2})^n & (1/\sqrt{2})^n & (1/\sqrt{2})^{n-1} & ... &
(1/\sqrt{2})^3 & (1/\sqrt{2})^2 & -1/\sqrt{2} & 0 & ....\\
.... \\
\end{array} \right)
\]
and thus $\rho(u)$ is a unitary element  of $q\Mul(\A )q$.
Define for all $k\ge0 $ $$f_k =:\rho( u^* e_k u).$$
Then $\{ f_k \}_{k =0}^{\infty}$ is a
sequence of pairwise orthogonal (equivalent) projections in $\A $ and $$\sum_{k=0}^\infty f_k = \rho (u) \big(\sum_{k=0}^\infty
q_k\big)\rho(u)^* = q.$$
Since $A_{00}\ge 0$ it follows from (\ref {e:A}) that
\begin{eqnarray}\label{e:5*}
 \| f_k A f_k \|
& \geq & \| f_k ( \sum_{j =0}^\infty \lambda_j q_j ) f_k \| -
\| f_k a f_k \| \notag \\
& = & \| \rho\Big (u^* e_k u \big(\sum_{j =0}^\infty\lambda_j e_j\big) u^* e_k u\Big) \| - \| f_k a f_k \| \notag \\
& = & \| e_k u\big( \sum_{j =0}^\infty \lambda_j e_j\big) u^* e_k \| -
 \| f_k a f_k \|\notag \\
& = & \Big( u \big(\sum_{j =0}^\infty \lambda_j e_j\big) u^*\Big)_{k,k} -
\| f_k a f_k \|.\notag
\end{eqnarray}
\emph{Claim 1:} $ \Big( u\big(\sum_{j =0}^\infty \lambda_j e_j\big) u^*\Big)_{k,k}
    > 1 +  (\lambda_0 - 1)/2^{k + 2}$ for all $k$.\\
For ease of computations, notice that
\begin{equation}\label{e:u}
u_{i,j} =
\begin{cases}
(1/\sqrt{2})^{i+1} & j = 0 \\
(1/\sqrt{2})^{i + 2 - j} & 1 \leq j \leq i \\
- 1/ \sqrt{2} & j = i + 1\\
0 & j > i+1
\end{cases}
\end{equation}
and  thus for all $i,j \geq 0$,
\begin{align}\label{e: 4*}
(u e_{0} u^*)_{i,i} &=|u_{i,0}|^2= 1/ 2^{i +1 }\\
| u_{i,j} |& \leq (1/ \sqrt{2})^{i + 1 - j}\label{e: 5}
\end{align}

Then
\begin{alignat*}{2}
\Big( u\big(\sum_{j =0}^\infty \lambda_j e_j\big) u^*\Big)_{k,k}& =\Big( u\big( I+ (\lambda_0 - 1) e_0   + \sum_{j = 1}^\infty (\lambda_j - 1) e_j\big) u^*\Big)_{k,k}\\
& =  1+(\lambda_0 - 1)/ 2^{k+1}-| \sum_{i=1}^{k + 1} u_{k,i} (\lambda_i - 1) \overline{u_{k,i}} |  \qquad &(\text {by (\ref {e:u})}) \\
& \ge  1+(\lambda_0 - 1)/ 2^{k+1}-\sum_{i=1}^{k+1} | u_{k,i} |^2 | \lambda_i - 1 |  \\
& >  1+(\lambda_0 - 1)/ 2^{k+1}-\sum_{i=1}^{k+1} (1/ 2)^{k + 1 - i}
(\lambda_0 - 1) / 4^{i+1}  \qquad &(\text {by (\ref {e:lambda}) and (\ref {e: 5}}) \\
& >  1+(\lambda_0 - 1)/ 2^{k+1}-\frac{(\lambda_0 - 1)}{2^{k+3}}\\
&> 1 + (\lambda_0 - 1)/ 2^{k +2},
\end{alignat*}
which proves Claim 1.

\medskip
\noindent \emph{Claim 2:}  $\| f_k a f_k \| < (\lambda_0 - 1)/ 2^{k +2}$ for all $k$.

Indeed,
\begin{alignat*}{2}
 \| f_k a f_k \|
 & =  \| \rho(u^* e_k u) a \rho(u^*e_k u)\| \\
 & =  \|q_k \rho(u) a \rho(u)^*q_k \|\qquad &(\text{since $\rho(u)$ is unitary}) \\
& =  \|  q_k\rho(u)\big(\sum_{i = 0}^{\infty }q_i \big) a\big(\sum_{j = 0}^{\infty }q_j \big) \rho(u^*)q_k \| &(\text{since $\rho(u)=q\rho(u)=\rho(u)q$}) \\
& =  \| \sum_{i,j = 0}^{k+1} q_k\rho(u)q_i aq_j \rho(u^*)q_k \| &(\text{by \ref {e:u})}) \\
& \leq  \sum_{i,j=0}^{k+1} \|q_k \rho(u)q_i \|\, \|q_j \rho(u^*)q_k \|\,
\| q_i a q_j \| \\
& = \sum_{i,j=0}^{k+1} | u_{k,i}|\,|u_{k,j} | \, \| q_i a q_j \| \\
&\le   \sum_{i,j=0}^{k+1} | u_{k,i}|\,|u_{k,j} | \, ( \| q_i a_+ q_j \| + \| q_i a_- q_j \|) &(\text{since $a=a_+-a_-$)}) \\
& \le  \sum_{i,j=0}^{k+1} (1/\sqrt{2})^{k + 1 - i} (1/\sqrt{2})^{k + 1 -j}
 \frac { (\lambda_0-1) (\sqrt 2 - 1)^2 } {2^{i+j +2}} &(\text{by (\ref {e: 5}) and (\ref {e:apm})}) \\
&= \frac { (\lambda_0-1) (\sqrt 2 - 1)^2 } {2^{k +3}}  \sum_{i=0}^{k+1} (1/\sqrt{2})^i\sum_{j=0}^{k+1}(1/ \sqrt{2})^j\\
&= \frac { (\lambda_0-1) (\sqrt 2 - 1)^2 } {2^{k +3}} \frac{2}{(\sqrt 2 -1)^2}\big (1-(\frac{1}{\sqrt 2})^{k+2}\big)^2\\
& <  (\lambda_0 - 1)/ 2^{k +2},
\end{alignat*}
which proves Claim 2.

    From (\ref {e:5*}), Claim 1 and Claim 2, we see that for all $k$,
$$\| f_k A f_k \| > 1 + (\lambda_0 - 1)/ 2^{k +2} - (\lambda_0 - 1)/ 2^{k +2}
 = 1.$$
\end{proof}

\begin{rem}\label{keylemma}
The key idea in the proof for Lemma 4.4 above is that by acting
within a copy of $\mathbb{B}( l_2 ) $ that is identified with the
corner $q \Mul(\A )q$ of $\Mul(\A )$, the unitary matrix $u$ permits
to turn the diagonal operator $\sum_{k=0}^\infty \lambda _k q_k$
which has one entry larger than 1 and all the other entries ``close"
to 1 into an operator with all the diagonal entries strictly
larger than 1. This ``spreading out" action of $u$ can be
illustrated directly in $\mathbb{B}( l_2 ) $ by showing that if for
some $t>0$ we set $D:=I+te_o\in \mathbb{B}( l_2 )$, i.e., the
diagonal operator with diagonal sequence $$< 1+t, 1, 1,1,\cdots >,$$
then the diagonal sequence of $uDu^*$ is
$$< 1+1/2t, 1+1/2^2t, \cdots 1+ 1/2^nt,\cdots>$$
where indeed each diagonal entry is  larger than 1.

\end{rem}

 \begin{prop}\label{P:ess=1}
Let $\A$ be a $\sigma$-unital non-unital   purely infinite simple
$C^*$-algebra and let $A\in \Mul(\A)^+$ with $\|A\|_{ess}=1$ and $
\|A\|>1$. Then $A$ is a strict sum of projections.
 \end{prop}
 \begin{proof}
 Since $\| A \|_{ess} = 1$ and $\| A \| > 1$, there exists a sequence
$\{ f_k \}_{k=0}^{\infty}$  of pairwise orthogonal projections in
$\A  $ satisfying all conditions of Lemma \ref{lem:technicallemma}.

Let $p = \sum_{k=0}^{\infty} f_k$. Then $p$ is a projection of $ \Mul(\A )$. Rewrite $A$ as
$$A = A^{1/2}(1 - p)A^{1/2} + \sum_{k=0}^{\infty} A^{1/2} f_k A^{1/2}$$
where the sum converges in the strict topology in $\Mul(\A )$ and
 $\| A^{1/2} f_k A^{1/2} \| > 1$ for all $k \geq 0$ by condition (iii)
 of Lemma  \ref{lem:technicallemma}. Now choose a sequence of mutually orthogonal
 projections $g_k$ of $\A$ whose sum converges to the identity.
 Then \\$A^{1/2} (1-p)g_k(1-p) A^{1/2} \in \A^+$ for every $k$ and $$A^{1/2}(1 - p)A^{1/2} =
 \sum_{k=1}^\infty A^{1/2} (1-p)g_k(1-p) A^{1/2}.$$ in the strict topology.
 Let  $$a_k:= A^{1/2} (1-p)g_k(1-p) A^{1/2}+ A^{1/2} f_k A^{1/2}.$$
 Then $a_k\in \A^+$, $\|a_k\|\ge \|A^{1/2} f_k A^{1/2}\|> 1$ for al $k$ and
 $$
 A= \sum_{k=1}^\infty a_k
 $$
 in the strict topology. By Remark \ref {R:sums}, $A$ is a strict sum of projections.
 \end{proof}

This provides the last substantial step in the proof of our main theorem.
\begin{proof}[{\bf Proof of Theorem \ref {T:main}}]

First, the necessity. If $A$ is a strict sum of projections belonging to $\A$, then either the number of projections is finite, in which case $A\in \A$ (case (iv)), or it is infinite, in which case $\| A \|_{ess} \ge 1$ by  Lemma \ref {L:ess norm}. It is clear that $\| A \|  \geq 1$. If $\| A \|  = 1$, then $A$ must be itself a projection (case (iii)). Indeed it is well known that if $p, q$ are two projections and $\| p + q \| = 1$ then $p$ and $q$ must be orthogonal. Finally, if $\| A \|  > 1$ then we can have either $\| A \|_{ess} > 1$ (case (i)) or $\| A \|_{ess} =1$ (case(ii).)

Now the sufficiency.
\item[(i)] Proposition \ref {P:ess>1}
\item[(ii)] Proposition  \ref {P:ess=1}
\item[(iii)] If $A=p\in \Mul(\A))$ is a projections, then $p\A p$ has an increasing approximate identity of projections (\cite{Zhang5}), say $f_n$, and hence, $A=p=\sum
_{n=1}^\infty (f_n-f_{n-1})$ as a strict sum of projections of $\A $
(where $f_0$).
\item[(iv)] Nothing to prove.
\end{proof}

\begin{rem}\label{final}
We would like to  point  out that due to \cite[Theorem 1.1] {B(H)case}, the operator $A=I+ (1/2)e\in
 \mathbb{B}( l_2 ) $, where $e$ is a rank one projection, cannot be written as a  strongly convergent sum of projections in $\mathbb{B}( l_2 ) $. However, if we unitarily embed $ \mathbb{B}( l_2 ) $
 in the multiplier algebra $\Mul(\A )$ where $\A $ is a $\sigma$-unital,
nonunital purely infinite simple C*-algebra, then A can be written as a strictly convergent sum of projections in $\A $. This is due to the much richer structure of  $\Mul(\A )$ than  $\mathbb{B}( l_2 )$.
\end{rem}



\end{document}